\documentclass[reqno, 11pt]{jams-l}

\usepackage{amssymb}
\usepackage{amsthm}
\usepackage[colorlinks,citecolor=red,pagebackref,hypertexnames=false]{hyperref}
\usepackage{color}
\usepackage{esint}
\usepackage{bbm}
\usepackage{graphicx}
\usepackage{mathrsfs}
\usepackage[left=3.3cm, right=3.3cm, top=2.7cm, bottom=2.7cm]{geometry}
\usepackage{enumitem}
\usepackage[hang,flushmargin]{footmisc}
\usepackage{comment}

\def\dA{{\mathcal{A}}}
\def\dB{{\mathcal{B}}}

\def\dF{{\mathcal{F}}}
\def\dG{{\mathcal{G}}}
\def\dN{{\mathcal{N}}}
\def\dS{{\mathcal{S}}}

\def\bZ{{\mathbb{Z}}}

\def\cC{{\mathscr{C}}}

\def\ve{\epsilon} 

\renewcommand{\d}{{\partial}}
\def\lec{\lesssim}

\DeclareMathOperator{\diam}{diam}
\def\dim{\mathop\mathrm{dim}} 					
\def\dist{\mathop\mathrm{dist}} 						

\newcommand{\ps}[1]{\left( #1 \right)}

\newcommand{\ck}[1]{\left\{#1 \right\}}

\def\XXint#1#2#3{{\setbox0=\hbox{$#1{#2#3}{\int}$ }
\vcenter{\hbox{$#2#3$ }}\kern-.58\wd0}}

\newtheorem{theorem}{Theorem}[section]
\newtheorem{lemma}[theorem]{Lemma}

\theoremstyle{definition}
\newtheorem{definition}[theorem]{Definition}

\theoremstyle{remark}
\newtheorem{remark}[theorem]{Remark}

\numberwithin{equation}{section}

\newcommand{\R}{\mathbb{R}}
\newcommand{\N}{\mathbb{N}}
\newcommand{\Z}{\mathbb{Z}}

\newcommand{\hd}{\mathcal{H}^d}

\newcommand{\hdc}{\mathcal{H}^d_\infty}

\newcommand{\B}{\mathbb{B}}

\newcommand{\spt}{\mathrm{spt}}

\newcommand\blfootnote[1]{%
  \begingroup
  \renewcommand\thefootnote{}\footnote{#1}%
  \addtocounter{footnote}{-1}%
  \endgroup
}

\newcommand{\dH}{\mathcal{H}}

\numberwithin{equation}{section}
\theoremstyle{plain}

\newtheorem{proposition}[theorem]{Proposition}

\newtheorem{claim}[theorem]{Claim}

\newcommand{\dD}{\mathcal{D}}

\newcommand{\cc}{\mathsf{c}}

\newcommand{\aA}{\mathsf{A}}

\newcommand{\Tree}{\mathsf{Tree}}
\newcommand{\Stop}{\mathsf{Stop}}
\newcommand{\Top}{\mathsf{Top}}

\newcommand{\Ss}{\mathsf{S}}

\newcommand{\doi}[1]{\textsc{doi}: \href{http://dx.doi.org/#1}{\nolinkurl{#1}}}

\makeatletter
\def\@makefnmark{%
  \leavevmode
  \raise.9ex\hbox{\fontsize\sf@size\z@\normalfont\tiny\@thefnmark}}
\makeatother

\newcommand{\RomanNumeralCaps}[1]
    {\MakeUppercase{\romannumeral #1}}

\title[Hausdorff dimension of non flat sets with PBP]{A note on the Hausdorff  dimension of uniformly non flat sets with plenty of big projections}

\author{Michele Villa}
\address{Research Unit of Mathematical Sciences, University of Oulu. P.O. Box 8000, FI-90014, University of Oulu, Finland.}
\email{michele.villa "at" oulu.fi}

\date{}

\dedicatory{}

\begin{document}
\maketitle

\begin{center}
\begin{minipage}[c][][r]{400pt}
\begin{small}
\textsc{Abstract.} 
Using a recent result of Orponen \cite{orponen2021plenty}, we show that sets with plenty of big projections (PBP) admit an Analyst's Travelling Salesman Theorem. We then show that sets with PBP which are uniformly non-flat (or wiggly) have large Hausdorff dimension. We also obtain a corollary on analytic and harmonic Lipschitz capacities. 
\end{small}
\end{minipage}
\end{center}  
\blfootnote{$3^{rd}$ of November 2021\\\textup{2010} \textit{Mathematics Subject Classification}: \textup{28A75}, \textup{28A12} \textup{28A78}.
 M. V. was supported by the Academy of Finland via the project Incidences on Fractals, grant No. 321896.
}

\tableofcontents

\section{Introduction}
In this note we record some consequences of a recent result by T. Orponen \cite{orponen2021plenty}, which states that an Ahlfors regular set with \textit{plenty of big projections} (PBP) has \textit{big pieces of Lipschitz graph} (BPLG). Let us give some definitions. Given an integer $1 \leq d\leq n$, a set $E \subset \R^n$ is called Ahlfors $d$-regular if there exists a constant $C \geq 1$ so that
\begin{align*}
	C^{-1} r(B)^d \leq \dH^d(B\cap E) \leq C r(B)^d,
\end{align*}
for each ball $B$ centered on $E$ with $r(B) \leq \diam(E)$. Here $r(B)$ denotes the radius of $B$. A set $E$ is said to have plenty of big $d$-dimensional projections ($d$-PBP) if there exist $\delta, \ve> 0$ so that the following hold. For each ball $B$ centered on $E$, there exists a $d$-plane $V_B \in \dG(n,d)$ so that
\begin{align}\label{e:PBP}
    \hd(\pi_V(E\cap B)) \geq \delta r(B)^d \mbox{ for all } V \in B(V_B, \ve).
\end{align}
Here $\pi_V$ is the orthogonal projection $\R^n \to V$, $\dG(d,n)$ is the Grassmannian manifold of $d$-dimensional (linear) planes in $\R^n$, and the ball $B(V_0, \ve)$ is defined with respect to the standard metric on $\dG(n,d)$.  Finally, a set $E$ has big pieces of Lipschitz graph if there are positive constants $\theta, L$ so that for every ball $B$ centered on $E$ with $r(B)\leq \diam(E)$, there exists an $L$-Lipschitz graph $\Gamma$ so that 
\begin{align}\label{e:BPLG}
	\dH^d(B \cap E \cap \Gamma) \geq \theta r(B)^d. 
\end{align}
Orponen's work proves a conjecture by David and Semmes from 1993 \cite{david1993quantitative} (see also \cite{martikainen2018characterising}). This conjecture was formulated in the context of \textit{quantitative rectifiability}, see \cite{david-semmes91, david-semmes93}. While the theory was originally developed for Ahlfors regular sets, in recent years several works have extended it to more general settings.  This extension has taken two forms, distinguished by flavour and aims.  A first thread of papers\footnote{For a more exaustive overview on the literature, see \cite{villa2019higher}.} (\cite{jones90, schul2007, badger2015multiscale, azzam2018analyst, villa2019higher, azzam2019quantitative, hyde2020TST}) focused on the geometry of sets and the generalisations have most often been stated in terms of \textit{Analyst's Travelling Salesman theorems} (TST). A second thread (\cite{girela2018riesz, dkabrowski2021measures, tolsa2021measures}) deals mostly with Radon measures and have kept its interest on the harmonic analytic side. This note belongs to the first thread - \textit{we prove an Analyst's TST for sets with PBP}. But because the precise statement entails introducing some technical tools, we first give two interesting corollaries.

\subsection{Applications: dimension of uniformly non flat sets}
\begin{theorem}\label{t:corollary}
Let $n \geq 2$ and $n-1 \geq d \geq 1$. Let $E \subset \R^n$ be a closed set with $d$-PBP \textup{(}with parameters $\ve,\delta>0$\textup{)} and which is uniformly non $d$-flat \textup{(}with constant $\beta_0$\textup{)}. Then 
\begin{equation*}
    {\rm dim}_H(E) \geq d + c \gamma(\ve, \delta) \beta_0^2,
\end{equation*}
where $c$ depends on $n,d$ and $\gamma(\ve, \delta)\to 0$ as $\ve \to 0$ or $\delta \to 0$.
\end{theorem}
\noindent
A  set $E$ is \textit{uniformly  non flat} of dimension $d$ and with parameter $\beta_0$ if for all balls $B$ centered on E it holds that
\begin{align}\label{e:beta-jones}
 \beta_{E}^{d,1}(B) > \beta_0.
 \end{align}
Here $\beta_E^{d,1}$ is an averaged version of the well-known Jones' coefficients (see \cite{jones90}). See Definitions \ref{d:jones-beta} and \ref{d:content-beta}.
 Uniformly non flat sets, also known as uniformly wiggly sets, were first studied by Bishop and Jones in \cite{bishop1997wiggly}, were they proved Theorem \ref{t:corollary} with $E$ a connected set in the plane, and $d=1$. Their motivation was the study of the Hausdorff dimension of  limit sets of certain groups. A result in this vein was later proved by David \cite{david2004hausdorff} (and quantified in \cite{villa2019higher}), this time for uniformly non flat sets of \textit{any integer} dimension, satisfying a topological condition. David's work was motivated by a question of L. Potyagailo, concerning higher dimensional limit sets (see the introduction of \cite{david2004hausdorff}). A motivation for publishing this note is that the PBP condition is in principle easier to check than David's topological one, hence we expect the result to be more applicable, at least in this context.
 
 \begin{remark}[On the definition of non-flatness]
 Uniform non-flatness could be defined in terms of $\beta_E^{d,p}$, for $p>1$ or $\beta_E^{d, \infty}$. In the first case one would obtain the same result as in Theorem \ref{t:corollary}, as long as $1 \leq p < p(d)$, where $p(d)$ is defined in Theorem \ref{t:main} - the reason for this is simply Theorem \ref{t:main} holds in this range, and it is a major tool in the proof of the dimension estimate \eqref{e:beta-jones}. On the other hand, if we define uniform non-flatness with respect to $\beta_E^{d,\infty}$ we obtain that 
 \begin{align}\label{e:dim-bound2}
 	{\rm dim}_H(E) \geq d+ c \gamma(\ve, \delta)\beta_0^{2(d+1)}
 \end{align}
for $d>1$, since for such $d$'s Theorem \ref{t:main} doesn't hold with $\beta_E^{d,\infty}$. In the case $d=1$, we still have the bound ${\rm dim}_H(E)>1 + c \gamma(\ve, \delta) \beta_0^2$, which is better than \eqref{e:dim-bound2}.
 \end{remark}

\begin{remark}[On the constant $\gamma(\delta,\ve)$]
	In a first draft of this note, the Theorem \ref{t:corollary} was wrongly stated: the dimension estimated did not depend on the PBP parameters, as the constant $\gamma(\ve, \delta)$ had been overlooked by the author. The dependence of $\dim(E)$ on the PBP parameters is not a proof artifact: consider a four corners Cantor set $E$, constructed in the usual way, except that we dilate by a small constant $\eta$ the squares in the construction. In the limit, we will obtain that a) $\dim_{H}(E) = 1 + C(\eta)$ with $C(\eta)$ as $\eta\to 1$, b) $\beta_E^{1,\infty}(B)\gtrsim 1$ for any ball $B$ centered on $E$, and c) $E$ will have PBP with parameters depending on $\eta$. This example contradicts the old (wrong) version of Theorem \ref{t:main}, where $\gamma(\ve, \delta)$ didn't appear. I thank T. Orponen for pointing out this issue and providing the example above.
\end{remark}

\subsection{Applications: analytic and Lipschitz harmonic capacities}
Below, if $d=1$, `Lipschitz harmonic' should be read `analytic'.  
For $d \geq 1$, let $E\subset \R^{d+1}$. We say that $E$ is \textit{removable for Lipschitz harmonic (LH) functions} if for any open set $\Omega \supset E$, any function $f:\Omega \to \R$ which is Lipschitz in $\Omega$ and harmonic in $\Omega\setminus E$ can be extended to the whole of $\Omega$. To understand what sort of sets are removable, the Lipschitz harmonic capacity was introduced:
\begin{equation*}
	\kappa(E):= \sup |\langle \Delta f, 1 \rangle|, 
\end{equation*}
where the supremum is taken over all Lipschitz functions on $\R^{d+1}$ which are harmonic in $\R^{d+1} \setminus \R$ and  satisfy $\|\nabla f\|_\infty \leq 1$, and where $\Delta f$ is understood in the sense of distribution. It is true that $E$ is removable if and only if $\kappa(E)=0$. Using \cite[Corllary 1.5]{tolsa2021measures} and Theorem \ref{t:main} we obtain the following quantitative bound for subsets of sets with PBP.
\begin{theorem}\label{t:corollary-2}
	Let $\Sigma\subset \R^{d+1}$ be a compact set with $d$-PBP with parameters $\delta, \ve>0$ and $\hd(E)<+\infty$. Then for any compact subset $E \subset \Sigma$, 
	\begin{align*}
			\kappa(E) \geq C \gamma(\ve, \delta)^{\frac{1}{2}} \frac{\hdc(E)^{\frac{3}{2}}}{\hd(\Sigma)^{\frac{1}{2}}},
	\end{align*}
where $C$ only depends on $d$ and $\gamma(\ve, \delta)$ is as in Corollary \ref{t:corollary}.  
\end{theorem}
When $d=1$ and $\Sigma$ is a rectifiable graph, Corollary \ref{t:corollary-2} was shown by Murai in \cite{murai1987comparison}. It was Verdera who noted that the result can be obtained via Menger curvature and Jones' travelling salesman theorem, see \cite[Theorem 4.31]{tolsa-book}. To the author knowledge, the results is new for Lipschitz harmonic capacity. See however fundamental previous related works \cite{paramonov1990harmonic}, \cite{mattila1995geometric} and \cite{volberg2003calderon}. See the books \cite{tolsa-book} and \cite{dudziak2011vitushkin} on the subject. The proof of Corollary \ref{t:corollary-2} follows Verdera's, see Section \ref{s:cap}.

The theorem below is the Analyst's TST which gives the two applications above. 
\begin{theorem}\label{t:main}
Let $E\subset \R^n$ be a set with $d$-PBP with constants $\ve, \delta>0$, $\dD$ be a system of Christ-David cubes, $Q_0 \in \dD$ and $C_0$ a sufficiently large constant. Let $1 \leq p < \tfrac{2d}{d-2}$ (and $1 \leq p \leq \infty$ if $d=1$).
 Then
\begin{align}\label{e:main}
    {\rm diameter}(Q_0)^d + \beta(Q_0) \approx \hd(Q_0),
\end{align}
where the constant behind the symbol $\approx$ depends on the constant $\ve, \delta$, $\aA$, $p$, $n,d$ and on the constants from the Azzam-Schul TST (see Theorem A.1 in \cite{azzam2019quantitative}).
\end{theorem}
\noindent
Here
\begin{align}\label{e:BETA}
    \beta(Q_0) = \beta_{E, \aA, p, d} (Q_0): = \sum_{Q \in \dD(Q_0)} \beta_E^{d,p} (C_0 B_Q)^2 \ell(Q)^d
\end{align}
and the coefficients $\beta_E^{d,p}$ are a variant of Jones' $\beta$ numbers \eqref{e:beta-jones} introduced by Azzam and Schul in \cite{azzam2018analyst}. We will define these coefficients and the Christ-David cubes in the preliminaries section below.  We refer the reader to Sections 1 and 3 of \cite{villa2019higher} for an overview on Analyst's TST and their relevance. 

\begin{remark}
	What we really prove here is one direction of the inequality (bound on the $\beta$ sum with the measure). See Proposition \ref{l:main-lemma}, where also the constant $\gamma(\ve, \delta)$ in Corollaries \ref{t:corollary} and \ref{t:corollary-2} appears.
\end{remark}

\subsection*{Acknowledgments}
I thank D. Dabrowski, M. Hyde, T. Orponen and X. Tolsa for useful comments and discussions on the manuscript and/or on related topics. In particular, T. Orponen pointed out that a previous version of Corollary \ref{t:corollary} was incorrectly stated, and suggested some counterexamples. The application in Corollary \ref{t:corollary-2} was suggested in discussions with Dabrwoski and Tolsa at the conference `Rajchman, Zygmund, Marcinkiewicz' at IM PAN. I also thank the organisers for the kind hospitality. 

\section{Preliminaries}
\subsection{Notation} \label{sec:notation}
We gather here some notation and some results which will be used later on.
We write $a \lesssim b$ if there exists a constant $C$ such that $a \leq Cb$. By $a \sim b$ we mean $a \lesssim b \lesssim a$.
In general, we will use $n\in \N$ to denote the dimension of the ambient space $\R^n$, while we will use $d \in \N$, with $d\leq n-1$, to denote the dimension of a subset $E \subset \R^n$.

For two subsets $A,B \subset \R^n$, we let
$
    \dist(A,B) := \inf_{a\in A, b \in B} |a-b|.
$
For a point $x \in \R^n$ and a subset $A \subset \R^n$, 
$
    \dist(x, A):= \dist(\{x\}, A)= \inf_{a\in A} |x-a|.
$
We write 
$
    B(x, r) := \{y \in \R^n \, |\,|x-y|<r\},
$
and, for $\lambda >0$,
$
    \lambda B(x,r):= B(x, \lambda r).
$
At times, we may write $\B$ to denote $B(0,1)$. When necessary we write $B_n(x,r)$ to distinguish a ball in $\R^n$ from one in $\R^d$, which we may denote by $B_d(x, r)$. 
We denote by $\dG(n,d)$ the Grassmannian, that is, the manifold of all $d$-dimensional linear subspaces of $\R^n$. A ball in $\dG(n,d)$ is defined with respect to the standard metric
\begin{align*}
	d_{\dG}(V, W) = \|\pi_V - \pi_W\|_{{\rm op}}.
\end{align*}
Recall that $\pi_V: \R^n \to V$ is the standard orthogonal projection onto $V$.
With $\dA(n,d)$ we denote the affine Grassmannian, the manifold of all affine $d$-planes in $\R^n$.

\subsection{Christ-David and dyadic cubes}
The family of dyadic cubes in $\R^n$ will be denoted by $\Delta$, the family of dyadic cubes with sidelength $\ell(I)=2^{-k}$ by $\Delta_k$.
\begin{theorem}[\cite{christ1990b, david-wavelets, hytonen2012non}] \label{theorem:christ}
Let $X$ be a doubling metric space and $X_k$ be a sequence of maximal $\rho^k$-separated nets, where $\rho = 1/1000$ and let $c_0 = 1/500.$ Then, for each $k \in \Z$, there is a collection $\dD_k$ of cubes such that the following hold.
\begin{enumerate}
\item For each $k \in \Z, \ X = \bigcup_{Q \in \dD_k}Q.$
\item If $Q_1, Q_2 \in \dD = \bigcup_k \dD_k$ and $Q_1 \cap Q_2\neq \emptyset$, then $Q_1 \subset Q_2$ or $Q_2 \subset Q_1$.
\item For $Q \in \dD$, let $k(Q)$ be the unique integer so that $Q \in \dD_k$ and set $\ell(Q)=5 \rho^k$. Then there is $x_Q \in X_k$ such that 
\begin{align*}
B(x_Q,c_0\ell(Q)) \subseteq Q \subseteq B(x_Q , \ell(Q)).
\end{align*}
\end{enumerate}
\end{theorem}

The $k^{th}$-\textit{generation children} of $Q \in \mathcal{D}$, denoted by $\text{Child}_k(Q),$ are the cubes $R \subseteq Q$ so that $\ell(R) = \rho^k\ell(Q).$ We also need the notion of a stopping-time region. 

\begin{definition}\label{d:ST}
A collection of cubes $S \subseteq \mathcal{D}$ is called a \textit{stopping-time region} if the following hold.
\begin{enumerate}
\item There is a cube $Q(S) \in S$ such that $Q(S)$ contains all cubes in $S$. 
\item If $Q \in S$ and $Q \subseteq R \subseteq Q(S),$ then $R \in S$.
\item If $Q \in S$, then all siblings of $Q$ are also in $S$. 
\end{enumerate}
\end{definition}

\subsection{Choquet integration and $\beta$-numbers}
For $1 \leq p<\infty$ and $A \subset \R^n$ Borel, we define the $p$-Choquet integral as 
\begin{align*}
    \int_A f(x)^p\, d \hdc(x) := \int_0^\infty \hdc(\{x \in A\, |\, f(x)>t\}) \, t^{p-1}\, dt.
\end{align*}
 We refer the reader to \cite{mattila} for more detail on Hausdorff measures and content and to Section 2 and the Appendix of \cite{azzam2018analyst} for more details on Choquet integration.  
 
\begin{lemma}\label{l:jensen}
Let $E \subseteq \R^n$ be either compact or bounded and open so that $\hd(E) >0,$ and let $f \geq 0$ be continuous on $E$. Then for $1 < p \leq \infty,$
\[ \frac{1}{\hd_\infty(E)} \int_E f \, d\hd_\infty \lesssim_n \left( \frac{1}{\hd_\infty(E)} \int_E f^p \, d\hd_\infty \right)^\frac{1}{p} \] 
\end{lemma}
\noindent
We recall the various $\beta$-numbers that we will use and state some of their properties. 
\begin{definition}[Jones]\label{d:jones-beta}
Let $E \subseteq \R^n$ and $B$ a ball. Define
\begin{align*}
\beta_{E,\infty}^d(B) = \frac{1}{r_B}\inf_L \sup\{\text{dist}(y,L) : y \in E \cap B\}
\end{align*}
where $L$ ranges over $d$-planes in $\R^n.$ 
\end{definition}
\begin{definition}
	Let $\mu$ be a measure on $\R^n$ with $\mu(B) \lesssim r(B)^d$, $p\geq1$. Then set
	\begin{equation*}		
		\beta_\mu^{d,p}(B) := \inf_L\left( \frac{1}{r(B)^d} \int_B \left( \frac{\dist(y, L)}{r(B)}\right)^p \, d\mu(y) \right)^{\frac{1}{p}}.
\end{equation*} 
\end{definition}
\begin{definition}[Azzam-Schul]\label{d:content-beta}
Let $1 \leq p < \infty,$ $E \subseteq \R^n$ and $B$ a ball. For a $d$-dimensional plane $L$ define
\begin{align} \label{e:beta-content} \beta^{d,p}_E(B,L) = \left( \frac{1}{r_B^d} \int_{E \cap B} \left( \frac{\dist(y,L)}{r_B} \right)^p \, d\mathcal{H}^d_\infty(y) \right)^\frac{1}{p}.
\end{align}
Then $\beta_{E}^{d,p}(B) = \inf_{L} \beta_{E}^{d,p}(B,L)$, where, again, the infimum is over all affine planes $L \in \dA(n,d)$. 
\end{definition}
\noindent
These coefficients and their variants have found several applications (\cite{villa2020tangent, hyde2021d, hyde2021cone})
We will need the following lemma.
\begin{lemma}[{\cite{azzam2018analyst}, Lemma 2.21}] \label{lemma:azzamschul}
Let $1 \leq p < \infty$ and $E_1, E_2 \subset \R^n$. Let $x \in E_1$ and fix $r>0$. Take some $y \in E_2$ so that $B(x,r) \subset B(y, 2r)$. Assume that $E_1, E_2$ are both lower content $d$-regular with constant $c$. Then
\begin{align*}
   \beta_{E_1}^{d,p} (x,r) \lesssim_c \beta_{E_2}^{d,p} (y, 2r) + \left( \frac{1}{r^d} \int_{E_1 \cap B(x, 2r)} \left(\frac{\dist (y, E_2)}{r} \right)^p \, d \hdc(y)\right)^{\frac{1}{p}}.
\end{align*}
\end{lemma}

\section{Proof of Theorem \ref{t:main}}
In this section we prove Theorem \ref{t:main}.

\subsection{Lower content regularity and coronisation}
Recall that a set $E \subset \R^n$ is lower content $(d, \cc_1)$-regular if, for all balls $B$ centered on $E$, 
\begin{align*}
    \hdc(E \cap B) \geq \cc_1 r(B)^d. 
\end{align*}
A nice fact about lower content regular set is that they admit a coronisation by Ahlfors regular sets. We first show that sets with $d$-PBP are in fact lower content regular, and the describe the corona construction.
\begin{lemma}\label{l:low-reg-E}
Let $E \subset \R^n$ be a set with $d$-PBP with constants $\delta, \ve>0$. Then $E$ is lower content $d$-regular with constant $c\sim_d \delta$.
\end{lemma}

\begin{proof}
Without loss of generality, we identify $V$ with $\R^d$. For an arbitrary $\ve_1>0$, let $\dB$ be a family of balls in $\R^d$ so that $\sum_{B' \in \dB} r(B')^d \leq  \hd (\pi_V(E \cap B)) - \ve_1$. Note that, since these are balls in a $d$-plane, $\hd(B' \cap \pi_V(B\cap E)) \lesssim_d r(B')^d$. Let $\ve, \delta$ be the parameter with which $E$ satisfied $d$-PBP. Fix a ball $B$ centered on $E$, with $r(B) \leq \diam(E)$, and a plane $V$ in $B(V_B, \ve)$. Then, 
\begin{align*}
    \delta r(B)^d \leq & \hd(\pi_V(E \cap B))\\
    & \leq \sum_{B' \in \dB} \hd(\pi_V(E \cap B) \cap B') \lesssim_d \sum_{B' \in \dB} r(B')^d \leq C\hd(\pi_V(E \cap B)) + C\ve_1 \\
    & \quad \quad \leq C'\hdc(\pi_V(E \cap B)) + C\ve_1. 
\end{align*}
Now, since $\pi_V$ is $1$-Lipschitz and $\ve_1$ was arbitrary, we obtain the lemma. The lower content regularity constant $c$ depends only on $\delta$ and $d$, since $C$ in the above display only depends on $d$. 
\end{proof}

\subsection{Discrete approximation} \label{s:ER}
In this subsection we recall the corona construction from \cite{azzam2019quantitative}. 
\begin{lemma}[\cite{azzam2019quantitative}, Main Lemma] \label{l:corona}
Let $k_0>0$, $\tau>0$, $d>0$ and $E \subset$ be a closed subset that is lower content $(d, \cc_1)$-regular. Let $\dD_k$ denote the Christ-David cubes on $E$ of scale $k$ and $\dD=\bigcup_{k\in\bZ} \dD_{k}$. Let $Q_{0}\in \dD_{0}$ and $\dD(Q_0, k_0)=\dD(k_0)=\bigcup_{k=0}^{k_0}\{Q\in \dD_{k}|Q\subseteq Q_0\}$. Then, for\footnote{$\Top(k_0)$ is a sub-collection of Christ-David cubes.} $R \in \Top(k_{0})\subseteq \dD(k_{0})$, we may partition $\dD(k_{0})$ into stopping-time regions which we call $\Tree(R)$; this partition has the following properties:
\begin{enumerate}[leftmargin=0.8cm]
\item There exists a constant $\eta=\eta(d, \cc_1)$ so that
\begin{equation}
\label{e:ADR-packing}
\sum_{R \in \Top(k_{0})} \ell(R)^{d} \leq \eta^{-1} \dH^{d}(Q_0),
\end{equation}
and $\eta\to 0$ as $\cc_1 \to 0$\footnote{This is not explicitly stated in \cite{azzam2019quantitative}, but it can be deduced from the proof, specifically see (3.4), (3.5) and (3.10) there.}.
\item Given $R\in \Top(k_{0})$ and a stopping-time region $\Ss \subseteq \Tree(R)$ with maximal cube $T=T(\Ss)$, let  $\dF=\dF(\Ss)$ denote the minimal cubes of $\Ss$ and set
\begin{align}\label{e:d_F}
    d_{\Ss, T, \dF}(x) = d_{T}(x) := \inf_{Q \in \dF} \ps{ \ell(Q) + \dist(x,Q)}.
\end{align}
For $C_{0}>4$ and $\tau>0$, there is a collection  $\cC(T(\Ss))=\cC_T \subset \Delta$ of disjoint dyadic cubes covering $C_{0}B_{T}\cap E$ so that 
if 
\[
E(T(\Ss))=E_T:=\bigcup_{I\in \cC_T} \d_{d} I,\]
where $\d_{d}I$ denotes the $d$-dimensional skeleton of $I$, then the following hold:
\begin{enumerate}[label=\textup{(}\alph*\textup{)}, leftmargin=0.8cm]
\item $E_T$ is Ahlfors $d$-regular with constants depending on $C_{0},\tau,d,$ and $\cc_1$.
\item We have the containment
\begin{equation}
\label{e:contains}
C_{0}B_{T}\cap E \subseteq \bigcup_{I\in \cC} I\subseteq 2C_{0}B_{T}.
\end{equation}

\item $E$ is close to $E_T$ in $C_{0}B_{T}$ in the sense that
\begin{equation}
\label{e:adr-corona}
\dist(x,E_T)\lec  \tau d_{T}(x) \;\; \mbox{ for all }x\in E\cap C_{0}B_{T}.
\end{equation}
\item The dyadic cubes in $\cC_T$ satisfy
\begin{equation}
\label{e:whitney-like}
\ell(I)\approx \tau \inf_{x\in I} d_{T}(x) \mbox{ for all }I\in \cC_T.
\end{equation}
\end{enumerate}
\end{enumerate}
\end{lemma}

\begin{remark}\label{r:notation1}
We will apply Lemma \ref{l:corona} with $\Ss=\Tree(R)$, so with the notation of the lemma, $R=T(\Ss)$. So we also put $E(T(\Ss))=E(R)=:E_R$.
\end{remark}

\begin{lemma}\label{l:approx-BPMD}
Let $E \subset \R^n$ be a set with $d$-PBP. Let $E_R$ be one of the sets from Lemma \ref{l:corona}. Then $E_R$ also has $d$-PBP \textup{(}with possibly a slightly different constant $\delta'$  which equals to $\delta$ up to a multiplicative constant depending only of $d$\textup{)}.
\end{lemma}

\begin{proof}
Let $x \in E_R$. By definition, there exists a dyadic cube $I \in \cC_R$ so that $x \in \partial_d I$. Suppose first that $r \geq 10 \ell(I)$. Since $I \cap E \neq \emptyset$, there exists a point $y \in E$ with $|x-y| \lesssim_n \ell(I)$. Thus we can find a radius $r' \approx r$ and so that $r' \gtrsim_n \ell(I)$ and for which $B(y, r') \subset B(x, r)$. 
Thus 
\begin{align*}
    \hdc(\pi_V(E_R \cap B(x,r))) & \approx \hdc\left( \pi_V\left( \cup_{I \in \cC_R(x,r)} I \right) \right)\\
    & \gtrsim \hdc\left( \pi_V (E \cap B(y,r'))\right) \gtrsim \delta (r')^d \approx \delta r^d.  
\end{align*}
Moreover, since $\pi_V(E_R \cap B(x,r))$ is a subset of a $d$-plane, its Hausdorff content is comparable to its Hausdorff measure. This proves the lemma for $r\geq c \ell(I)$. If $r< c \ell(I)$, then we simply note that for all radii $r' \geq r$, $E_R$ has large $d$-projections in plenty of directions because it's a union of $d$-dimensional planes. 
\end{proof}

\subsection{Beta numbers estimates}
\begin{remark}
In view of Corollary \ref{t:corollary}, we will keep careful track of the dependence of the various constants on $\ve$ and $\delta$ (the parameteres from the PBP condition). On the other hand, we will often `forget' dependence on $n,d, \tau$, which are fixed throughout. Hence, in all estimates, $\lesssim$ is allowed to depend on $n,d, \tau$, but \textit{not} on $\delta, \ve$.
\end{remark}
In this subsection we will prove the following.
\begin{proposition}\label{l:main-lemma}
Let $n\in \N$, $n-1\geq d\geq 1$. If $E \subset \R^n$ has $d$-PBP with parameters $\ve, \delta>0$, then for any $Q_0 \in \dD(E)$, we have
\begin{align*}
    \ell(Q_0)^d + \sum_{Q \in \dD_E(Q_0)} \beta_E^{d,2}(3B_Q)^2 \ell(Q)^d \lesssim_{n,d} \gamma(\ve, \delta)^{-1} \, \hd(Q_0),
\end{align*}
where $\gamma(\ve, \delta) \to 0$ as $\ve \to 0$ or $\delta\to 0$ \textup{(}or both\textup{)}.
\end{proposition}
Theorem \ref{t:main} follows immediately from this proposition.
\begin{proof}[Proof of Theorem \ref{t:main}]
First, by \cite[(A.3)]{azzam2019quantitative}, we have that $\sum_{Q \subset Q_0} \beta_E^{d,p}(C_0 B_Q)^2 \,\ell(Q)^d \approx_{p, C_0} \sum_{Q \subset Q_0} \beta_E^{d,2}(3B_Q)^2 \ell(Q)^d$ whenever $C_0 >1$ and $1 \leq p \leq p(d)$. The converse inequality in Theorem \ref{t:main} follows from \cite[Theorem A.1(1)]{azzam2019quantitative}.
\end{proof}
We now focus on proving Lemma \ref{l:main-lemma}. We start off by 
applying Lemma \ref{lemma:azzamschul} with $E_1=E$ and $E_2 = E_{R}$. Let us check that with this choice the hypotheses are satisfied.
For $Q \in \dD$, recall that $z_Q$ denotes the center of $Q$. By the definition of $ \Tree(R)$, we see that if $Q \in \Tree(R)$, then there must exists a dyadic cube $I \in \cC_R$ which meets $Q$ (by Lemma \ref{l:corona}(2.b)). By \eqref{e:whitney-like}, $\ell(I) \lesssim \tau \ell(Q)$. Hence 
\begin{align}\label{e:XQ}
    \mbox{we find a point } \quad  x_Q \in E_{R} \quad \mbox{ such that } \quad |z_Q - x_Q| \leq 4 \tau \ell(Q),
\end{align}
and we obtain that
\begin{align}\label{e:form500}
    B_Q := B(z_Q, \ell(Q)) \subset B(x_Q, 2 \ell(Q))=: B'_Q.
\end{align}
 This implies that for each cube $Q \in  \Tree(R)$ the hypotheses of Lemma \ref{lemma:azzamschul} are satisfied (with $E_1 = E$ and $E_2 = E_{R, \rho}$). We then can estimate
 \begin{align*}
    \sum_{\substack{Q \in \Tree(R)\\ Q \in \dD(Q_0, k_0)}} \beta_{E}^{d, 2} (3 B_Q)^2 \ell(Q)^d  & \lesssim \sum_{\substack{Q \in \Tree(R)\\ Q \in \dD(Q_0,k_0)}} \beta_{E_{R}}^{d,2} (6 B_Q')^2 \, \ell(Q)^d \\
    & + \sum_{\substack{Q \in  \Tree(R)\\ Q \in \dD(Q_0,k_0)}} \ps{\frac{1}{\ell(Q)^d} \int_{6 B_Q\cap E} \ps{\frac{\dist(y, E_{R})}{\ell(Q)}}^2 \, d \hdc(y) }\, \ell(Q)^d \\
    & := I_1 + I_2.
\end{align*}
We estimate $I_1$. We denote by $\dD(E_R)$ a family of Christ-David cubes for $E_R$ obtained by applying Theorem \ref{theorem:christ} to $E_R$. Using \eqref{e:XQ} it is immediate to see that to each $Q' \in \dD(E_R)$, there correspond a bounded number (depending on $n, d$ and maybe $\tau$) of cubes $Q \in \Tree(R)$ with $\ell(Q) \approx \ell(Q')$ and so that \eqref{e:form500} holds. Thus we have that 
\begin{align*}
    I_1 \lesssim \sum_{Q \in \dD(E) \atop \ell(R) \lesssim \ell(R)} \beta_{E_R}^{d,2}(3 B_Q)^2 \ell(Q)^d 
\end{align*}
Since $E_R$ is Ahlfors $d$-regular (Lemma \ref{l:corona}) with constant $C$ depending on $C_0, \tau, d$ and $\cc_1$ (and thus, by Lemma \ref{l:low-reg-E}, on $\delta$) and has $d$-PBP  with parameters $(\delta, \ve)$, then, by the main result of \cite{orponen2021plenty}, it satisfies the \textit{weak geometric lemma}, that is 
\begin{align*}
    \sum_{Q: \beta_E^\infty(Q)> \ve_1} \ell(Q)^d \leq C(\ve_1; \ve, \delta) \ell(Q_0),
\end{align*}
where, if $\ve, \delta$ are fixed, $C(\ve_1; \ve, \delta) \to \infty$ as $\ve_1 \to 0$, but also, for fixed $\ve_1$, $C(\ve_1; \delta,\ve)\to\infty$ as $\delta\to 0$ or $\ve\to 0$. This and the fact that $E_R$ has (trivially from the PBP condition) one big projection, gives, via \cite[Theorem 1.14]{david1993quantitative}, that $E_R$ has BPLG (recall \eqref{e:BPLG}).
In particular, it is uniformly rectifiable. It follows from the strong geometric lemma (see \cite{david-semmes91}, Section 15) that 
\begin{equation}
    I_1 \lesssim C(\delta, \ve) \ell(R)^d,
\end{equation}
where $C(\delta, \ve) \to \infty$ as $\ve\to 0$ or $\delta\to 0$ (or both). \\

\noindent
We now now estimate $I_2$. 
Let $y \in 6B_R$. Recall that $\dF(R):= \dF(\Tree(R))$ (see Lemma \ref{l:corona}(2)) are the minimal cubes in $\Tree(R)$. For a given $S \in \dF(R)$, there is a dyadic cube $I \in \cC_R$ with $I \cap S \neq \emptyset$ and $\rho\ell(S) \leq \ell(I) \leq \ell(S)$. Note then that if $y \in S\in \dF(R)$, then
\begin{align}
    \dist(y, E_R) \lesssim \ell(S) + \dist(y', E_R) \stackrel{\eqref{e:adr-corona}, \eqref{e:whitney-like}}{\lesssim} \ell(S),
\end{align}
with $y' \in S \cap I$.
We write
\begin{align*}
   I_2 \lesssim \sum_{\substack{Q \in \Tree(R)\\ Q \in \dD(Q_0, k_0)}} \int_{6 B_Q\cap E} \ps{\frac{\dist(y, E_{R})}{\ell(Q)}}^2 \, d \hdc(y)
   &  \lesssim \sum_{\substack{Q \in \Tree(R)\\ Q \in \dD(Q_0,k_0)}}\sum_{\substack{S \in \dF(R) \\ S\cap 6B_Q \neq \emptyset }}  \frac{\ell(S)^{2+d}}{\ell(Q)^2}.
\end{align*}
We now swap the sums (which are all finite), to obtain that
\begin{align}
    I_2 & \lesssim \sum_{\substack{S \in \dF(R)\\ S \cap 6B_{Q_0} \neq \emptyset}}  \ell(S)^{d + 2} \sum_{\substack{Q \in  \Tree(R) \\ \exists Q' \in \dN(Q):\, Q' \supset S}}  \frac{1}{\ell(Q)^{ 2}}
    & \lesssim \sum_{\substack{S \in \Stop(R) \\ S \cap 6B_{Q_0} \neq \emptyset }} \ell(S)^{d + 2} \sum_{\substack{Q \in  \Tree(R) \\ \exists Q' \in \dN: \, Q' \supset S}} \frac{1}{\ell(Q)^{2}}, \label{e:estBeta1}
\end{align}
where $\dN(Q):=\{Q' \in \dD(k_0) \, |\, \ell(Q')=\ell(Q) \mbox{ and } Q' \subset 6B_Q \}$.
The number of cubes $Q \in  \Tree(R)$ of a given generation so that there exists a cube $Q' \in \dN(Q)$ for which $Q' \supset S$ is $\lesssim_n 1$. The interior sum in \eqref{e:estBeta1} is then a geometric series and
\begin{align*}
    \sum_{\substack{Q \in  \Tree(R) \\ \exists Q' \in \dN(Q):\, Q' \supset S}} \frac{1}{\ell(Q)^{ 2}} \lesssim_{n} \frac{1}{\ell(S)^{2}}. 
\end{align*}
Therefore we obtain 
\begin{align*}
    \eqref{e:estBeta1} \lesssim \sum_{\substack{S \in \dF(R) \\ S \cap 6B_{Q_0} \neq \emptyset}}  \frac{\ell(S)^{d +2 }}{\ell(S)^{2}} = \sum_{\substack{S \in \Stop(R) \\ S \cap 6B_{Q_0} \neq \emptyset}} \ell(S)^d.
\end{align*}
This latter sum is bounded above by $ \lesssim_\delta \ell(R)^d$ since $\dF(R)$ is a disjoint family of cubes and $\dF(R) \subset \Tree(R)$. We have proved the following claim. 
\begin{lemma}
\begin{align}\label{e:TreeEst}
\sum_{\substack{Q \in \Tree(R) \\ Q \in \dD(Q_0, k_0)}} \beta_E^{d,2}(3B_Q)^2 \ell(Q)^d \lesssim C(\ve, \delta) \ell(R)^d,
\end{align}
where $C(\ve, \delta) \to \infty$ as $\delta\to 0$ or $\ve \to 0$.
\end{lemma}
We conclude that
\begin{align}
    \sum_{Q \in \dD(Q_0, k_0)} \beta_{E}^{d,2}(3B_Q)^2 \ell(Q)^d & \stackrel{{\rm Lemma} \ref{l:corona}}{\lesssim} \sum_{R \in \Top(k_0)} \sum_{Q\in \Tree(R)} \beta_{E}^{d,2}(3B_Q)^2 \ell(Q)^d  \notag \\ & \stackrel{\eqref{e:TreeEst}}{\lesssim}  C(\ve, \delta) \sum_{R \in \Top(k_0)} \ell(R)^d  \stackrel{\eqref{e:ADR-packing}}{\lesssim}  C(\ve, \delta)\eta^{-1} \hd(Q_0) \notag\\
    & \qquad \qquad \qquad \qquad \qquad \qquad \qquad \approx C(\ve, \delta) \hd(Q_0) \label{e:upper-bound}
\end{align}
This prove the upper bound in Theorem \ref{t:main}. The lower bound follows immediately from the main result of Azzam and Schul \cite{azzam2018analyst}, which says that
\begin{equation*}
    \hd(Q_0)\lesssim \ell(Q_0)^d + \sum_{Q \subset Q_0} \beta_E^{d, p}(3 B_Q)^2 \ell(Q)^d.
\end{equation*}
Here the implicit constant depends on $n, \cc_1 \sim \delta$ and the (fixed) parameteres appearing in \cite[Theorem A.1]{azzam2019quantitative}. See Theorem A.1. in \cite{azzam2019quantitative} for details.
So we are done with the proof of Theorem \ref{t:main}. 

\section{Proof of Corollary \ref{t:corollary}}
\begin{remark}
	The strategy of the proof is that of Bishop and Jones \cite{bishop1997wiggly}. We use extra care to keep track of the constant $\gamma(\delta, \ve)$.
\end{remark}
In this appendix we prove Corollary \ref{t:corollary}
Let $E \subset \R^n$ be a set with PBP, with constants $\delta, \ve>0$ (see \eqref{e:PBP}). For $N_0 \in \N$, fix a cube $R \in \dD_{N_0}(E)$. For an integer $m \geq N_0$, define
\begin{align*}
    \beta_m(R) =  \sum_{Q \in \dD_m(R)} \beta_E^{d,2}(Q) \ell(Q)^d. 
\end{align*}
Let $c<1$ be a constant of the form $2^{-s}$, $s \in \N$, to be fixed later. It will depend on $\rho>0$ from Theorem \ref{theorem:christ}. Set
\begin{align} \label{e:Delta-ck}
    \Delta_{k,c}(R) :=  \ck{ I \in \Delta \, |\, I \cap R \neq \emptyset \mbox{ and } \ell(I) = c\,2^{-k}},
\end{align}
and
\begin{align*}
    E_{R, k} := \bigcup_{I \in \Delta_{k,c}(R)} \partial_d I. 
\end{align*}

\noindent
\begin{lemma}There exists a constant $\gamma= \gamma(\ve, \delta)$, with $\gamma \to 0$ as $\delta\to 0$ or $\ve\to 0$ (or both), so that, if $R \in \dD_{N_0}(E)$
with $N_0 \leq k$, then
\begin{equation} \label{e:BJ-a}
    \sum_{m=N_0}^k \beta_m (R) + \ell(R)^d \leq C\gamma^{-1}\hd(E_{R, k}),
\end{equation}
where $C$ depends only on $d$ and $n$.
\end{lemma}
\begin{proof}
As it was done in Lemma \ref{l:approx-BPMD}, one might show that if $E$ has $d$-PBP with constants $\ve, \delta$, then so does $E_{R, k}$, with constants $\ve, \delta'>0$, where $\delta'$ is equal to $\delta$ up to a multiplicative constant depending only on $d$. Then, we apply Theorem \ref{t:main} (or, more specifically, \eqref{e:upper-bound}), and Lemma \ref{l:low-reg-E}, to obtain
\begin{align*}
  \ell(R)^d + \sum_{Q \in \dD_{E_{R,k}}} \beta_{E_{R, k}}^{d,2}(Q)^2 \ell(P)^d \lesssim_{n, d} C(\ve, \delta) \hd(E_{R, k}). 
\end{align*}
where $C(\ve, \delta)$ is as in \eqref{e:upper-bound}. The claim \eqref{e:BJ-a} now follows quickly: consider a cube $Q \in \dD_{E}$, such that $\ell(Q) > c 2^{-k}$, for $c<1$ as in \eqref{e:Delta-ck}. If we choose $c$ sufficiently small, we can apply Lemma \ref{lemma:azzamschul} with $E_1=E$ and $E_2 = E_{R,k}$, to obtain
\begin{align*}
    \beta_{E}^{d,p}(C_0P) \lesssim \beta^{d,p}_{E_{R,k}}(6P) + \ps{ \frac{1}{\ell(P)^d} \int_{6 B_P} \ps{\frac{\dist(y, E_{R,k})}{\ell(P)}}^p \, d \mathcal{H}_\infty^d}^{\frac{1}{p}}.
\end{align*}
Thus we see that
\begin{align}
    & \sum_{\substack{P \in \dD_{E}\\ \ell(P) > c 2^{-k}}} \beta_{E}^{d,p} (3 P) \notag \\
    & \lesssim \sum_{\substack{P \in \dD_{E_{R, k}} \\ \ell(P) \gtrsim c 2^{-k}}} \beta^{d,p}_{E_{R,k}}(6 P) + \sum_{\substack{P \in \dD_{E} \\ \ell(P) > c 2^{-k}}} \ps{ \frac{1}{\ell(P)^d} \int_{6 B_P} \ps{\frac{ \dist(y, E_{R,k})}{\ell(P)}}^{p} \, \hdc(y) }^{\frac{1}{p}}.\label{e:form101}
\end{align}
With a calculation similar to that from \eqref{e:estBeta1}, we obtain that the second sum above is $ \lesssim \ell(R)^d$. This then gives 
\begin{align*}
    2 \hd(E_{R, k}) 
    & \gtrsim \ell(R)^d + \frac{1}{C(\ve, \delta)} \left( \ell(R)^d + \sum_{P \in \dD_{E_{R,k}}} \beta_{E_{R, k}}^{d,p}(C_0 P)^2 \ell(P)^d \right) \\
    &  \gtrsim \frac{1}{C(\ve, \delta)} \Bigg[ \ell(R)^d + \sum_{P \in \dD_{E_{R,k}} \atop \ell(P') \geq c 2^{-k}} \beta_{E_{R,k}}^{d,p}(C_0 P)^2 \ell(P)^d \\
    & \quad \quad + \sum_{P \in \dD_{E} \atop \ell(P) > c 2^{-k}} \ps{ \frac{1}{\ell(P)^d} \int_{2C_0 B_P} \ps{\frac{ \dist(y, E_{R,k})}{\ell(P)}}^{p} \, \hdc(y)}^{\frac{1}{p}} \Bigg]\\
    & \stackrel{\eqref{e:form101}}{\gtrsim} \frac{1}{C(\ve, \delta)} \left( \ell(R)^d + \sum_{P \in \dD_{E} \atop \ell(P) \gtrsim c 2^{-k}} \beta_{E}^{d,p}(C_0P)^2 \ell(P)^d \right). 
\end{align*}
This proves \eqref{e:BJ-a} with $\gamma(\ve, \delta)\sim  C(\ve, \delta)^{-1}$.
\end{proof}

\begin{claim}
Let $N$ an integer so that $N > N_0$ (recall that $N_0$ is the scale of $R$, i.e. $R \in \dD_{N_0}(E)$). Consider a dyadic cube $I_N \in \Delta_N(\R^n)$ for which $\ell(I_N) < \ell(R)/10$ and such that $\frac{1}{3} I_N \cap E \neq \emptyset$. For $k > N$, we have
\begin{equation}\label{e:form102}
    \sum_{m=N}^k \sum_{Q \in \dD_m(R) \atop Q \cap I_N \neq \emptyset} \beta_{E}^{d,2}(3 Q)^2 \ell(Q)^2  \geq C_1 C(\delta) (k-N) \beta_0^2 2^{-dN},
\end{equation}
where $C_1$ is a constant depending only on $n, d$ and maybe on $c$ from \eqref{e:Delta-ck}, but not on $\delta\sim \cc_1$. On the other hand, $C(\delta) \to 0 $ as $\delta \to 0$. 
\end{claim}

\begin{proof}
To see this, note first that by lower $(d, \cc_1)$-regularity of $E$ (with $\cc_1 \sim_d \delta)$), there are  $\gtrsim C(\delta) 2^{d(m-N)}$ dyadic cubes $J$ of generation $m$ (with $m>N$) such that $J \subset I_N$ and $J \cap E \neq \emptyset$. Clearly, $C(\delta)\to0$ as $\delta\to 0$. Hence since $E$ is uniformly non-flat, we see that if $N \leq m \leq k$,  
\begin{align*}
     \sum_{Q \in \dD_m(R) \atop Q \cap I_N \neq \emptyset} \beta_{E}(3 Q)^2 \ell(Q)^d 
    & \geq  \beta_0^2 \sum_{Q \in \dD_m(R) \atop Q \cap I_N} \ell(Q)^d \\
    & \approx_c \beta_0^2 \sum_{J \in \Delta_{m,c}(R)\atop J \subset I_N} \ell( J)^d 
     \gtrsim  C(\delta) \beta_0^2 2^{d(m-N)} 2^{-dm}  \approx_{n,d, c} C(\delta) \beta_0^2 2^{-dN}.
\end{align*}
This gives \eqref{e:form102}. 
\end{proof}

Now, using \eqref{e:BJ-a}, we immediately obtain
\begin{align*}
    \hd(E_{R,k}\cap I_N) \gtrsim \gamma(\ve, \delta) (k-N) \, \beta_0^2 \, 2^{-dN},
\end{align*}
where we incorporated $C(\delta)$ into $\gamma(\ve, \delta)$. 
Let now $\ck{z_j}$, $j$ in some index set $A$, be a maximal $2^{-k}$-separated net of $E_{R,k}\cap I_N$ such that 
$ 
    \bigcup_{j \in A} B(z_j, 2^{-k+2}) \supset E_{R,k} \cap I_N.
$
Then
\begin{align*}
    \hd(E_{R,k} \cap I_N) \lesssim 2^{-dk}\, \text{Card}(A).
\end{align*}
Thus we obtain 
$
     2^{-dk} \, \text{Card}(A) \gtrsim \gamma(\ve, \delta) (k-N) \, \beta_0^2 2^{-dN}, 
$
and therefore
\begin{align} \label{e:CardA}
    \text{Card}(A) \geq C_1 \gamma(\ve, \delta) (k-N) \, \beta_0^2\, 2^{d(k-N)}.
\end{align}
Since $k$ was an arbitrary integer with $k \geq N$, we can choose it so that
\begin{align*}
 \kappa:= k-N \approx C_1^{-1} \gamma(\ve, \delta)^{-1} \beta_0^{-2}.
\end{align*}
In particular, with $c = C_1/2$,
\begin{equation*}
\kappa C_1 \gamma(\ve, \delta) \beta_0^2 \geq 2^{\kappa\,  c \gamma(\ve, \delta) \beta_0^2} 
\end{equation*}
Hence we see from \eqref{e:CardA} that
\begin{align} \label{e:CardA-b}
    \text{Card}(A) \geq 2^{(d+C_1'\gamma(\ve, \delta)\beta_0^2)\kappa},
\end{align}
\noindent
We now apply this construction recursively for each $N > N_0$, as follows. For $N_0$, we put 
\begin{align*}
    \dS_0 := \ck{ I \in \Delta_{N_0+\kappa}(R) \, |\, \exists j \in A \mbox{ s.t. } z_j \in I}
\end{align*}
Then for each  $I \in \dS_0$, we find a maximal net $\{z_j\}_{j \in A}$ as above; the cardinality of this net will be again as in \eqref{e:CardA-b}. We put the relative cubes in the subfamily
\begin{align*}
    \dS(I):= \ck{ J \in \Delta_{N_0 + 2\kappa} \, |\, \exists j \in A \mbox{ s.t. }z_j \in J}.
\end{align*}
We then put
$
    \dS_1:= \bigcup_{I \in \dS_0} \dS(I).
$
Having defined $\dS_{j-1}$, we set
$
    \dS_{j} := \bigcup_{I \in \dS_{j-1}} \dS(I), 
$
where $\dS(I)=\{ j \in \Delta_{N_0 + j\kappa} \, |\, \exists j \in A \mbox{ s.t. } z_j \in J\}$.
Let us record that for each $j \in \N$, we have 
\begin{enumerate}
    \item Each $J \in \dS_j$, is a subset of some $I \in \dS_{j-1}$.
    \item Each $I \in \dS_{j-1}$ contains at least $2^{(d+ c \gamma(\ve, \delta) \beta_0^2) \kappa}$ cubes $I \in \dS_{j}$ (as in \eqref{e:CardA}).
    \item For each $j \in \N$, if $I \in \dS_j$, we have $I\cap R \neq \emptyset$.
\end{enumerate}

\noindent
\begin{lemma}
If $R$ satisfies (1)-(3), then 
\begin{align*}
    {\rm dim}_H(R) > d+ c \gamma(\ve, \delta)\beta_0^2.
\end{align*}
\end{lemma} 
\begin{proof}
To prove this claim, we define the $\mu$ on the elements $I$ of $\dS_j$, for $j \geq 0$, by 
\begin{align*}
    \mu(I) = \text{Card}(A)^{-j} \leq 2^{-j\kappa(d+c\gamma(\ve, \delta)\beta_0^2)}.
\end{align*}
One can then check that $\spt(\mu)= E$ and that $\mu(R) = 1$. Then, by Frostman's Lemma (Theorem 8.8 in \cite{mattila}), we have that
\begin{align*}
    \mathcal{H}^{d+ c \gamma(\ve, \delta)\beta_0^2}(R) >0.
\end{align*}
\end{proof}
This completes the proof of Corollary \ref{t:corollary}.

\section{Proof of Corollary \ref{t:corollary-2}}\label{s:cap}
Let $\Sigma \subset B_0\subset \R^{d+1}$ be a compact set with $d$-PBP with parameters $\delta, \ve>0$, let $E \subset \Sigma$. By Frostman's lemma (see \cite[Remark 1.24]{tolsa-book}), there exists a measure $\mu$ with $\mu(B) \leq C r(B)^d$ supported on $E$ so that $\mu(E) \approx \dH_\infty^d(E)$, were the implicit constant only depends on $d$. Then, using Lemma \ref{l:low-reg-E}, we have, for any ball centered on $\spt(\mu) \subset E \subset \Sigma$
\begin{align*}
	\beta_\mu^{d,2}(B) \lesssim \beta_\Sigma^{d,2}(B).
\end{align*}
Let $\Theta_\mu(B)= \mu(B)/r(B)^d$ be the $d$-density of $\mu$ in the ball $B$; since $\mu(B) \leq r(B)^d$, $\Theta_\mu(B) \leq 1$. By Theorem \ref{t:main}, we then get
\begin{align*}
	\beta^2(\mu, B) &:= \int_0^{r(B)} \int_B \beta_\mu^{d,2}(x,r)^2 \Theta_\mu(B(x,r)) \, d\mu(x) \, \frac{dr}{r} \\
	&  \lesssim \sum_{Q \in \dD_\mu \atop B_Q \subset 3B} \beta_\mu^{2,d}(B_Q)^2 \Theta_\mu(B_Q) \mu(Q) \\
	& \lesssim \sum_{Q\in \dD_\Sigma} \beta_\Sigma^{2, d}(3B_Q)^2 \ell(Q)^d \leq C_1 \gamma(\ve, \delta)^{-1} \hd(\Sigma \cap 3B).
\end{align*}
Here $C_1$ allowed to depend only on $d$.
This in particular holds for $B=B_0$. Set
\begin{equation*}
	C_2:= \left( \frac{\hd(E)}{C_1 \gamma(\ve, \delta)^{-1} \hd(\Sigma)} \right)^{\frac{1}{2}}.
\end{equation*}	
Then define the measure $\sigma := C_2 \mu$. This has clearly growth $\lesssim r^d$. We also have 
\begin{equation*}
	\beta^2(\sigma, B_0) \leq C_2^3 \beta^2(\mu, B_0)  \leq  \left( \frac{\hd(E)}{C_1 \gamma(\ve, \delta)^{-1} \hd(\Sigma)} \right)^{\frac{3}{2}} C_1 \gamma(\ve, \delta)^{-1} \hd(\Sigma).
\end{equation*}
Then, by \cite[Corollary 1.4]{tolsa2021measures}, 
\begin{align*}
	&\kappa (E) \geq \sigma(E) = C_2 \mu(E) = \left( \frac{\mu(E)}{C_1 \gamma(\ve, \delta)^{-1} \hd(\Sigma)} \right)^{\frac{1}{2}} \mu(E) \\
	& \qquad \qquad \qquad \qquad \qquad = C \gamma(\ve, \delta)^{\frac{1}{2}} \frac{\mu(E)^{\frac{3}{2}}}{\hd(\Sigma)^{\frac{1}{2}}} \approx \gamma(\ve, \delta)^{\frac{1}{2}} \frac{\hdc(E)^{\frac{3}{2}}}{\hd(\Sigma)^{\frac{1}{2}}}
\end{align*}

\bibliography{bibliography}
\bibliographystyle{halpha-abbrv}
\end{document}